\pdfoutput=1
\RequirePackage{ifpdf}
\ifpdf 
\documentclass[pdftex]{sigma}
\else
\documentclass{sigma}
\fi

\numberwithin{equation}{section}

\newtheorem{Theorem}{Theorem}[section]
\newtheorem*{Theorem*}{Theorem}
\newtheorem{Corollary}[Theorem]{Corollary}
\newtheorem{Proposition}[Theorem]{Proposition}
\newtheorem*{Proposition*}{Proposition}
 { \theoremstyle{definition}
\newtheorem{Example}[Theorem]{Example}
\newtheorem{Remark}[Theorem]{Remark} }

\newcommand{\skal}[2]{\langle #1,#2\rangle}

\begin{document}


\newcommand{\arXivNumber}{1503.03740}

\renewcommand{\PaperNumber}{107}

\FirstPageHeading

\ShortArticleName{Geometry of $G$-Structures via the Intrinsic Torsion}

\ArticleName{Geometry of $\boldsymbol{G}$-Structures via the Intrinsic Torsion}

\Author{Kamil NIEDZIA{\L}OMSKI}

\AuthorNameForHeading{K.~Niedzia{\l}omski}

\Address{Department of Mathematics and Computer Science, University of \L\'{o}d\'{z},\\
 ul.~Banacha 22, 90-238 \L\'{o}d\'{z}, Poland}
\Email{\href{mailto:kamiln@math.uni.lodz.pl}{kamiln@math.uni.lodz.pl}}
\URLaddress{\url{http://www.math.uni.lodz.pl/~kamiln/}}

\ArticleDates{Received April 28, 2016, in f\/inal form October 31, 2016; Published online November 04, 2016}

\Abstract{We study the geometry of a $G$-structure $P$ inside the oriented orthonormal frame bundle ${\rm SO}(M)$ over an oriented Riemannian
manifold $M$. We assume that $G$ is connected and closed, so the quotient ${\rm SO}(n)/G$, where $n=\dim M$, is a normal homogeneous space and we equip ${\rm SO}(M)$ with the natural Riemannian structure induced from the structure on~$M$ and the Killing form of~${\rm SO}(n)$. We show, in particular, that minimality of $P$ is equivalent to harmonicity of an induced section of the homogeneous bundle ${\rm SO}(M)\times_{{\rm SO}(n)}{\rm SO}(n)/G$, with a Riemannian metric on~$M$ obtained as the pull-back with respect to this section of the Riemannian metric on the considered associated bundle, and to the minimality of the image of this section. We apply obtained results to the case of almost product structures, i.e., structures induced by plane f\/ields.}

\Keywords{$G$-structure; intrinsic torsion; minimal submanifold; harmonic mapping}

\Classification{53C10; 53C24; 53C43; 53C15}

\section{Introduction}

Existence of a geometric structure on an oriented Riemannian manifold is equivalent to saying that the structure group ${\rm SO}(n)$ of the oriented orthonormal frame bundle reduces to a certain subgroup $G$. For example, for $G$ equal
\begin{gather*}
{\rm SO}(m)\times {\rm SO}(n-m),\quad {\rm U}(n/2),\quad {\rm U}(n/2)\times 1,\quad {\rm Sp}(n/4){\rm Sp}(1),\quad {\rm G}_2,\quad {\rm Spin}(7)
\end{gather*}
we have almost product, almost Hermitian, almost contact, almost quaternion-K\"ahler, ${\rm G}_2$ and ${\rm Spin}(7)$ structures, respectively. It is natural to ask if the holonomy group of the Levi-Civita connection $\nabla$ is contained in given $G$. The list of possible irreducible Riemannian holonomies is limited by the Berger list~\cite{lb}.

On the other hand, the defect of the Levi-Civita connection to be a $G$-connection measures the intrinsic torsion $\xi$, which is the dif\/ference of $\nabla$ and a $G$-connection $\nabla^G$ (with torsion),
\begin{gather*}
\xi_XY=\nabla_XY-\nabla^G_XY.
\end{gather*}
If a $G$-structure is integrable, i.e., the intrinsic torsion vanishes, then the holonomy is contained in the structure group $G$. The study of possible intrinsic torsions, i.e., the decomposition of the space of intrinsic torsions into irreducible modules, was initiated by Gray and Hervella \cite{gh} in the case of almost Hermitian manifolds. Later, many authors considered other possible cases (see, for example, \cite{cg,chs,cs,dm,ffs,fg,mc,ms,an,vw}).

The other possible direction, initiated by Wood \cite{cw3} and generalized to the general case by Gonz{\'a}lez-D{\'a}vila and Mart{\'{\i}}n~Cabrera~\cite{gmc2}, is to consider dif\/ferential properties of intrinsic torsion induced by condition of harmonicity of the unique section of the associated homogeneous bundle. More precisely, a $G$-structure $P\subset {\rm SO}(M)$, where $G$ is closed and connected, induces the unique section $\sigma$ of the homogeneous associated bundle ${\rm SO}(M)/G={\rm SO}(M)\times_{{\rm SO}(n)} {\rm SO}(n)/G$,
\begin{gather*}
\sigma(\pi_{{\rm SO}(M)}(p))=[[p,eG]],\qquad p\in {\rm SO}(M),
\end{gather*}
where $\pi_{{\rm SO}(M)}$ is the projection in the orthonormal frame bundle ${\rm SO}(M)$. Then the decomposition $\mathfrak{so}(n)=\mathfrak{g}\oplus\mathfrak{m}$, where $\mathfrak{m}$ is the orthogonal complement to $\mathfrak{g}$ with respect to the Killing form, on the level of Lie algebras is reductive. Equip ${\rm SO}(M)/G$ with the natural Riemannian metric induced from the Killing form on $\mathfrak{m}$ and Riemannian metric $g$ on $M$. Then we say that a $G$-structure $P$ is harmonic if the section $\sigma$ is harmonic. The correspondence of the notion of harmonicity with the intrinsic torsion follows from the fact that the intrinsic torsion $\xi$ can be considered as a section of the bundle $T^{\ast}M\otimes\mathfrak{m}_P$, where $\mathfrak{m}_P$ is the adjoint bundle $P\times_{\operatorname{ad}G}\mathfrak{m}$. This follows from the observation that the $\mathfrak{m}$-component of the connection form $\omega$ with respect to the reductive decomposition $\mathfrak{so}(n)=\mathfrak{g}\oplus\mathfrak{m}$ can be projected to the tangent bundle $TM$ (see the next section for the details).

To seek for the `best' possible non-integrable $G$-structures we consider the third possible approach. Namely, we focus on the minimality of a $G$-structure in the oriented orthonormal frame bundle ${\rm SO}(M)$. It is not surprising that minimality is related with the harmonicity of a~$G$-structure. More generally, we use the concept of intrinsic torsion to obtain some results on the geometry of~$G$-structures. The idea comes from the results obtained by the author~\cite{kn} in the case of a single submanifold. We deal with the intrinsic and extrinsic geometry of a~$G$-structure. To be more precise, we def\/ine the Riemannian metric on ${\rm SO}(M)$ by inducing it from the Riemannian metric~$g$ on~$M$ and Killing form ${\bf B}$ on the structure group ${\rm SO}(n)$. It is interesting that the Levi-Civita connection on~$P$ depends on the $G$-connection and Levi-Civita connection $\tilde{\nabla}$, which comes from a certain Riemannian metric $\tilde{g}$ on~$M$. This metric $\tilde{g}$ depends on the intrinsic torsion and equals the pull-back of the Riemannian metric on the associated homogeneous bundle ${\rm SO}(M)/G$ with respect to the section $\sigma$.

The main emphasis is put on the minimality of a $G$-structure $P$ in the Riemannian struc\-tu\-re~${\rm SO}(M)$. The main theorem can be stated as follows.
\begin{Theorem*}
The following conditions are equivalent:
\begin{enumerate}\itemsep=0pt
\item[$1)$] $G$-structure $P$ is minimal in ${\rm SO}(M)$,
\item[$2)$] the induced section $\sigma\colon (M,\tilde{g})\to (N,\skal{\cdot}{\cdot})$ is a harmonic map, where $\tilde{g}$ is such that $\tilde{g}=\sigma^{\ast}\skal{\cdot}{\cdot}$.
\end{enumerate}
\end{Theorem*}
Notice, that these conditions are also equivalent, by the general fact concerning harmonic maps \cite{el}, to the minimality of $\sigma(M)$ in $N$. Moreover, the vanishing of the second fundamental form is the neccesary condition for the integrability of a $G$-structure.
\begin{Proposition*}
If a $G$-structure $P$ is integrable, i.e., the intrinsic torsion vanishes, then $P$ is totally geodesic in ${\rm SO}(M)$.
\end{Proposition*}

The article is organized as follows. In Sections~\ref{section2} and~\ref{section3} we recall the notion of the intrinsic torsion and state its main properties. The results in these sections are well known and can be found in the literature, for example, in~\cite{gmc2,cw2}.

In Section~\ref{section4} we introduce a tensor, which transfers the Riemannian metric $g$ to the mentioned above $\tilde{g}$. Its properties are crucial in the main considerations. With the Riemannian metric~$\tilde{g}$ and the horizontal distribution $\mathcal{H}'$ of the $G$-structure $P$ induced by the minimal $G$-connec\-tion~$\nabla '$, the projection $\pi_P\colon P\to M$ becomes the Riemannian submersion.

Section~\ref{section5} is the main section of the general outline and deals with the geometry of $G$-structures. Firstly, we consider intrinsic geometry focusing on the curvatures and, secondly, we consider extrinsic geometry, with the main result concerning minimality of a $G$-structure.

We end the article with some relevant examples. Examples considered here have been considered by other authors \cite{ggv,gd,hl}, so we only list them and state relevant conditions.

Throughout the paper we will use the following notation and identif\/ication: For any associated bundle $E=P\times_G S$ with the f\/iber $S$ and induced by the principal bundle $P(M,G)$ any element in $E$ will be denoted by $[[p,s]]$. Moreover, we have $\Gamma(E)\equiv C^{\infty}(P,S)^G$, where $C^{\infty}(P,S)^G$ is the space of equivariant functions $f\colon P\to S$, $f(pg)=g^{-1}f(p)$. The identif\/ication is the following
\begin{gather*}
\Gamma(E)\ni\sigma\longleftrightarrow f\in C^{\infty}(P,S)^G,\qquad \sigma(\pi_P(p))=[[p,f(p)]].
\end{gather*}

\section{Intrinsic torsion}\label{section2}

In this section we will review the basic facts concerning intrinsic torsion of a $G$-structure \cite{gmc2,cw2}.

Let $(M,g)$ be an oriented Riemannian manifold, ${\rm SO}(M)$ its oriented orthonormal frame bundle. Denote by $\omega$ the connection form induced by the Levi-Civita connection $\nabla$ on $M$. Let~$\mathcal{H}$ and~$\mathcal{V}$ be horizontal and vertical distributions on ${\rm SO}(M)$, respectively.

Assume that the structure group ${\rm SO}(n)$ reduces to a closed and connected subgroup $G$. Then the quotient~${\rm SO}(n)/G$ is a normal homogeneous space, i.e., the subspace $\mathfrak{m}=\mathfrak{g}^{\bot}\subset \mathfrak{so}(n)$ def\/ines~$\operatorname{ad}(G)$ invariant decomposition $\mathfrak{so}(n)=\mathfrak{g}\oplus \mathfrak{m}$, where $\mathfrak{g}$ is the Lie algebra of $G$ and the orthogonal part is taken with respect to the Killing form ${\bf B}$. Denote by $P\subset {\rm SO}(M)$ the reduced subbundle. The $\mathfrak{g}$-component $\omega_{\mathfrak{g}}$ of $\omega$ def\/ines a connection form on $P$. Denote by $\mathcal{H}'$ and $\mathcal{V}'$ the horizontal and vertical distributions on $P$ with respect to $\omega_{\mathfrak{g}}$, respectively.

Moreover, consider the following associated bundles -- the adjoint bundles over $M$: $\mathfrak{so}(n)_P=P\times_{\operatorname{ad}G}\mathfrak{so}(n)$, $\mathfrak{g}_P=P\times_{\operatorname{ad}G}\mathfrak{g}$ and $\mathfrak{m}_P=P\times_{\operatorname{ad}G}\mathfrak{m}$.

Notice that $\mathfrak{so}(n)_P$ can be realized as the bundle $\mathfrak{so}(M)$ of skew-symmetric endomorphisms of the tangent bundle $TM$ via the identif\/ication
\begin{gather*}
[[p,A]]\mapsto p^{-1}\cdot A\cdot p,
\end{gather*}
where we treat element $p\in P_x$ as a linear isomorphism $p\colon \mathbb{R}^n\to T_xM$. Analogously, we have the bundles $\mathfrak{g}(M)$ and $\mathfrak{m}(M)$. Moreover, $\mathfrak{so}(n)_P$ is isomorphic to the bundle $\mathcal{V}_{\rm equiv}$ over $M$ of equivariant vertical vector f\/ields on $P$. Namely, the map
\begin{gather*}
[[p,A]]\mapsto A^{\ast}_p
\end{gather*}
settles this isomorphism, where $A^{\ast}$ denotes the fundamental vertical vector induced by the matrix $A\in\mathfrak{so}(n)$.

Let
\begin{gather*}
\xi(X)=\omega_{\mathfrak{m}}\big(X^{h'}\big),
\end{gather*}
where $\omega_{\mathfrak{m}}$ denotes the $\mathfrak{m}$-component of $\omega$ and $X^{h'}$ is the horizontal lift with respect to $\mathcal{H}'$. Since
\begin{gather*}
\omega_{\mathfrak{m}}\big(X^{h'}_{pg}\big)=\omega_{\mathfrak{m}}\big(R_{g\ast}X^{h'}_p\big)
=\operatorname{ad}\big(g^{-1}\big)\omega_{\mathfrak{m}}\big(X^{h'}_p\big),
\end{gather*}
it follows that $\xi(X)\in \mathfrak{m}_P$. Section $\xi\in\Gamma(T^{\ast}M\otimes\mathfrak{m}_P)$ is called the {\it intrinsic torsion} of a $G$-structure $P$. By above identif\/ications $\xi(X)=\xi_X\in \mathfrak{m}(M)$. Let $\nabla '$ be the connection on $M$ induced by $\omega_{\mathfrak{g}}$ on $P$. We call $\nabla '$ the {\it minimal connection} of a $G$-structure. Intrinsic torsion $\xi$ has the following properties:
\begin{gather}\label{eq:intrinsictorsionproperties}
X^{h'}=X^h+(\xi_X)^{\ast}\qquad\textrm{and}\qquad \xi_XY=\nabla_XY-\nabla'_XY,
\end{gather}
for $X,Y\in \Gamma(TM)$. The decomposition $\mathfrak{so}(M)=\mathfrak{g}(M)\oplus\mathfrak{m}(M)$ implies the following relations between intrinsic torsion and the curvature tensor $R$ of $\nabla$:
\begin{gather}\label{eq:curvaturedecomp}
\begin{split}
&(\nabla_X\xi_Y)_{\mathfrak{g}}=[\xi_X,\xi_Y]_{\mathfrak{g}},\\
&(\nabla_X\xi_Y)_{\mathfrak{m}}=\nabla '_X\xi_Y+[\xi_X,\xi_Y]_{\mathfrak{m}},\\
&R(X,Y)_{\mathfrak{g}}=R'(X,Y)+[\xi_X,\xi_Y]_{\mathfrak{g}},\\
&R(X,Y)_{\mathfrak{m}}=\nabla '_X\xi_Y-\nabla '_Y\xi_X+[\xi_X,\xi_Y]_{\mathfrak{m}}-\xi_{[X,Y]}.
\end{split}
\end{gather}
Notice also that
\begin{gather}\label{eq:Rdecomp2}
R(X,Y)=R'(X,Y)+(\nabla_X\xi)_Y-(\nabla_Y\xi)_X-[\xi_X,\xi_Y],
\end{gather}
where $R'$ is the curvature tensor of $\nabla'$ and
\begin{gather*}
(\nabla_X\xi)_Y=\nabla_X\xi_Y-\xi_{\nabla_XY}.
\end{gather*}

\section[ Harmonic $G$-reductions]{Harmonic $\boldsymbol{G}$-reductions}\label{section3}

Let $(M,g)$ be an oriented Riemannian manifold, $\pi_{{\rm SO}(M)}\colon {\rm SO}(M)\to M$ its orthonormal frame bundle with the connection form $\omega$ inducing Levi-Civita connection $\nabla$ on $M$. Let $G\subset {\rm SO}(n)$, $n=\dim M$, be a subgroup such that ${\rm SO}(n)/G$ is a normal homogeneous space and let $P\subset {\rm SO}(M)$ the reduced subbundle. Let $N={\rm SO}(M)/G={\rm SO}(M)\times_{{\rm SO}(n)} ({\rm SO}(n)/G)$ be the homogeneous bundle associated with ${\rm SO}(M)$. Denote by $\zeta\colon {\rm SO}(M)\to N$ the natural projection. Then $\zeta$ def\/ines the $G$-principal bundle. Clearly $\zeta$ is constant on $P$ and hence there is a bijection between $G$-reductions of ${\rm SO}(M)$ and sections of $N$. Denote by $\sigma\in \Gamma(N)$ the section induced by~$P$. Let $\mathfrak{m}_N$ be the adjoint bundle associated with $\zeta$, i.e., $\mathfrak{m}_N={\rm SO}(M)\times_{\operatorname{ad}G}\mathfrak{m}$.

Let $\mathcal{H}^N$ and $\mathcal{V}^N$ be horizontal and vertical distributions on $N$, respectively, where $\mathcal{H}^N$ is induced by $\omega$. We have
\begin{gather*}
\mathcal{H}^N=\zeta_{\ast}\mathcal{H},\qquad \mathcal{V}^N={\rm SO}(M)\times_{{\rm SO}(n)} T({\rm SO}(n)/G),
\end{gather*}
where ${\rm SO}(n)$ acts on $T({\rm SO}(n)/G)$ by the dif\/ferential of the natural action of ${\rm SO}(n)$ on ${\rm SO}(n)/G$. Since $T({\rm SO}(n)/G)={\rm SO}(n)\times_{\operatorname{ad}G}\mathfrak{m}$ it follows that $\mathcal{V}^N$ is isomorphic to $\mathfrak{m}_N$.

Denote by $\varphi\colon TN\to \mathfrak{m}_N$ the following map
\begin{gather*}
\varphi([[p,[[g,A]]\,]])=[[p,A]],\qquad \varphi\big(X^{h,N}\big)=0,
\end{gather*}
i.e., $\varphi$ settles the described above isomorphism of $\mathcal{V}^N$ onto $\mathfrak{m}_N$ and is zero on the horizontal distribution. The Riemannian metric on $N$ is induced by $g$ and the Killing form on~$\mathfrak{m}$, namely,
\begin{gather}\label{eq:riemmetriconN}
\skal{V}{W}=g(\pi_{N\ast}V,\pi_{N\ast}W)+{\bf B}(\varphi(V),\varphi(W)),\qquad V,W\in TN,
\end{gather}
where ${\bf B}(A,B)=-\operatorname{tr}(AB)$ for $A,B\in\mathfrak{m}$. Decomposition
\begin{gather*}
TN=\mathcal{H}^N\oplus\mathcal{V}^N,
\end{gather*}
def\/ines projections ${\bf h}\colon TN\to\mathcal{H}^N$ and ${\bf v}\colon TN\to\mathcal{H}^N$. It can be shown that~\cite{gmc2}
\begin{gather}\label{eq:verticalsigma}
{\bf v}\sigma_{\ast}(X)=\xi_X\in\mathfrak{m}_N,\qquad X\in TM,
\end{gather}
where $\xi_X$ is the intrinsic torsion of a $G$-structure $P$.

In this section we will derive the formula for the derivative $\nabla\sigma_{\ast}$ and the tension f\/ield of~$\sigma$.

Equip $N$ with the Riemannian metric $\skal{\cdot}{\cdot}$ given by \eqref{eq:riemmetriconN} and~$M$ with a Riemannian metric~$\tilde{g}$ on~$M$, which may vary from~$g$. We will consider $\sigma$ as a map
\begin{gather*}
\sigma\colon \ (M,\tilde g)\to (N,\skal{\cdot}{\cdot}).
\end{gather*}
Assuming $M$ is compact we may def\/ine the energy functional $E(\sigma)$ given by
\begin{gather}\label{eq:energyfunctional}
E(\sigma)=\frac{1}{2}\int_M \|\sigma_{\ast}\|^2\,d{\rm vol}_M,
\end{gather}
where the norm $\|\cdot\|$ is taken with respect to $\tilde g$ and $\skal{\cdot}{\cdot}$, i.e.,
\begin{gather*}
\|\sigma_{\ast}\|^2=\sum_i\skal{\sigma_{\ast}(\tilde{e_i})}{\sigma_{\ast}(\tilde{e_i})},
\end{gather*}
where $(\tilde{e_i})$ is a $\tilde{g}$-orthonormal basis on $M$. In order to study harmonic sections it is convenient to study variations of the functional~\eqref{eq:energyfunctional} in the class of all sections of $N$. We will consider that above functional in the class of all maps from $M$ to $N$, thus we will study harmonicity as a map (not a section). Then, the Euler--Lagrange equation is given by the formula
\begin{gather}\label{eq:ELequation}
\tau_{\tilde{g}}(\sigma)=\operatorname{tr}_{\tilde g}\nabla\sigma_{\ast}
=\sum_i\nabla^{\sigma}_{\tilde{e_i}}\sigma_{\ast}\tilde{e_i}
-\sigma_{\ast}(\tilde{\nabla}_{\tilde{e_i}}\tilde{e_i})=0,
\end{gather}
where $\tilde{\nabla}$ is the Levi-Civita connection of $\tilde{g}$ and $\nabla^{\sigma}$ is the pull-back connection in the bundle
\begin{gather}\label{eq:TNdecomposition}
\sigma^{\ast}TN= \sigma^{\ast}\mathcal{H}^N\oplus\sigma^{\ast}\mathcal{V}^N =TM\oplus\mathfrak{m}_P=TM\oplus\mathfrak{m}(M).
\end{gather}
We call $\tau(\sigma)$ the {\it tension field} of $\sigma$. We say that a $G$-structure is {\it harmonic as a map} if~\eqref{eq:ELequation} holds for the induced section $\sigma$ (here and furthermore we do not require $M$ to be compact). Taking the decomposition of $\tau(\sigma)$ with respect to~\eqref{eq:TNdecomposition}, harmonicity of $\sigma$ is equivalent to vanishing of~${\bf h}\tau(\sigma)$ and~${\bf v}\tau(\sigma)$.

If $\tilde g=g$, then a $G$-structure induced by the section $\sigma$ satisfying ${\bf v}\tau(\sigma)=0$ is called a {\it harmonic} $G$-{\it structure}~\cite{gmc2}.

Denote by $\Pi_{\tilde g}$ and $\Pi_g$ the dif\/ferential $\nabla\sigma_{\ast}$ with respect to $\tilde g$ and $g$, respectively. Moreover, let $S$ be the dif\/ference between $\tilde{\nabla}$ and $\nabla$, i.e., $S(X,Y)=\tilde{\nabla}_XY-\nabla_XY$. Then
\begin{gather*}
\tau_{\tilde g}(\sigma)=\operatorname{tr}_{\tilde g}\Pi_{\tilde g}
=\operatorname{tr}_{\tilde g}(\Pi_g-\sigma_{\ast}S).
\end{gather*}
We have \cite{gmc2, cw2}
\begin{gather}
\varphi(\Pi_g(X,Y))=\frac{1}{2}((\nabla_X\xi)_Y+(\nabla_Y\xi)_X)\in\mathfrak{m}(M), \label{eq:vtau}\\
g(\pi_{N\ast}\Pi_g(X,Y),Z) =\frac{1}{2}(
{\bf B}(\xi_X,R_{\mathfrak{m}}(Y,Z))+{\bf B}(\xi_Y,R_{\mathfrak{m}}(X,Z))), \label{eq:htau}
\end{gather}
where $(\nabla_X\xi)_Y=\nabla_X\xi_Y-\xi_{\nabla_XY}$.

In order to describe explicitly the condition for the harmonicity of $\sigma$ let us introduce certain curvature operator~\cite{kn}. For any $\alpha\in\mathfrak{so}(M)$ let
\begin{gather}\label{eq:defRT}
R_{\alpha}(X)=\sum_i R(e_i,\alpha(e_i))X\in TM,\qquad X\in TM.
\end{gather}
Then $R_{\alpha}\in\mathfrak{so}(M)$ and the following formula holds
\begin{gather}\label{eq:R_Tproperty}
g(R_{\alpha}(X),Y)={\bf B}(\alpha,R(X,Y)),\qquad \alpha\in\mathfrak{so}(M),\qquad X,Y\in TM.
\end{gather}
Indeed,
\begin{gather*}
g(R_{\alpha}(X),Y) =\sum_i g(R(e_i,\alpha(e_i))X,Y)=\sum_i g(\alpha(e_i),R(X,Y)e_i) ={\bf B}(\alpha,R(X,Y)).
\end{gather*}

Now, we can state the formula for the harmonicity of $\sigma$.
\begin{Proposition}\label{prop:harmonicsigma1}
$G$-structure $\sigma\colon (M,\tilde{g})\to (N,\skal{\cdot}{\cdot})$ is a harmonic map if and only if the following conditions hold
\begin{gather*}
\sum_i (\nabla_{\tilde{e_i}}\xi)_{\tilde{e_i}}-\xi_{S(\tilde{e_i},\tilde{e_i})}=0 \qquad\textrm{and}\qquad
\sum_iR_{\xi_{\tilde{e_i}}}(\tilde{e_i})-S(\tilde{e_i},\tilde{e_i})=0,
\end{gather*}
where $(\tilde{e_i})$ is an orthonormal basis for $\tilde{g}$.
\end{Proposition}
\begin{proof}
Follows by \eqref{eq:verticalsigma}, \eqref{eq:vtau}, \eqref{eq:htau} and \eqref{eq:R_Tproperty}.
\end{proof}

\section{Properties of certain transfer tensor}\label{section4}

In this section we will introduce invertible tensor induced by the intrinsic torsion, the Riemannian metric def\/ined by this tensor and state the properties of the Levi-Civita connection of this new metric. Results in this section are generalizations of the results obtained by the author in \cite{kn}.

Adopt the notation from the previous section. For $\alpha\in \mathfrak{m}(M)$ put
\begin{gather*}
\xi\cdot \alpha=-\sum_i {\bf B}(\xi_{e_i},\alpha)e_i\in TM,
\end{gather*}
where we consider the intrinsic torsion as an element of $\mathfrak{m}(M)$. It is easy to show that
\begin{gather}\label{eq:xidotT}
g(\xi\cdot \alpha,X)=-{\bf B}(\alpha,\xi_X),\qquad \alpha\in\mathfrak{m}(M),\qquad X\in TM.
\end{gather}

Let $L$ be the endomorphism of the tangent bundle of the form
\begin{gather*}
L(X)=X-\xi\cdot \xi_X,\qquad X\in TM.
\end{gather*}
In order to derive some properties of $L$ recall the def\/inition of the Riemannian metric on the bundle ${\rm SO}(M)$ induced by the metric $g$ on $M$ and by the Killing form ${\bf B}$ on the structure group~${\rm SO}(n)$. We def\/ine the Riemannian metric $g_{{\rm SO}(M)}$ on~${\rm SO}(M)$ as follows
\begin{gather*}
g_{{\rm SO}(M)}\big(X^h,Y^h\big) =g(X,Y), \qquad
g_{{\rm SO}(M)}\big(X^h,\alpha^{\ast}\big) =0, \qquad
g_{{\rm SO}(M)}\big(\alpha^{\ast},\beta^{\ast}\big) ={\bf B}(\alpha,\beta),
\end{gather*}
where $X,Y\in TM$, $\alpha,\beta\in\mathfrak{so}(M)=\mathfrak{so}(n)_{{\rm SO}(M)}$. Recall that we identify element $\alpha\in \mathfrak{so}(M)$ with the equivariant vertical vector f\/ield denoted by $\alpha^{\ast}$. Then maps $\pi_{{\rm SO}(M)}\colon {\rm SO}(M)\to M$ and $\zeta\colon {\rm SO}(M)\to N$ are Riemannian submersions.

\begin{Proposition}\label{prop:PropertiesofL}
We have
\begin{gather*}
g_{{\rm SO}(M)}\big(X^{h'},Y^{h'}\big)=g(X,LY).
\end{gather*}
In particular, $L$ is a symmetric and positive definite automorphism of $TM$. Moreover, the covariant derivative of $L$ is related with the intrinsic torsion $\xi_X$ by the formula
\begin{gather*}
g((\nabla_XL)Y,Z)={\bf B}((\nabla_X\xi)_Y,\xi_Z)+{\bf B}((\nabla_X\xi)_Z,\xi_Y).
\end{gather*}
\end{Proposition}
\begin{proof}
By \eqref{eq:xidotT} and \eqref{eq:intrinsictorsionproperties},
\begin{gather*}
g(X,LY) =g(X,Y)+{\bf B}(\xi_X,\xi_Y)=g_{{\rm SO}(M)}\big(X^h+\xi_X,Y^h+\xi_Y\big) =g_{{\rm SO}(M)}\big(X^{h'},Y^{h'}\big).
\end{gather*}
Thus $L$ is symmetric and positive def\/inite. Moreover, by the fact that $\nabla$ is metric for ${\bf B}$ we have
\begin{gather*}
g((\nabla_XL)Y,Z)=g(\nabla_X(LY),Z)-g(L(\nabla_XY),Z)\\
\hphantom{g((\nabla_XL)Y,Z)}{} =Xg(LY,Z)-g(LY,\nabla_XZ)-g(\nabla_XY,LZ)\\
\hphantom{g((\nabla_XL)Y,Z)}{}=Xg(Y,Z)+X{\bf B}(\xi_Y,\xi_Z)-g(Y,\nabla_XZ)-{\bf B}(\xi_Y,\xi_{\nabla_XZ})\\
\hphantom{g((\nabla_XL)Y,Z)=}{}-g(\nabla_XY,Z)-{\bf B}(\xi_{\nabla_XY},\xi_Z)\\
\hphantom{g((\nabla_XL)Y,Z)}{}={\bf B}(\nabla_X\xi_Y-\xi_{\nabla_XY},\xi_Z) +{\bf B}(\nabla_X\xi_Z-\xi_{\nabla_XZ},\xi_Y).\tag*{\qed}
\end{gather*}\renewcommand{\qed}{}
\end{proof}

By Proposition \ref{prop:PropertiesofL} the symmetric and bilinear form
\begin{gather*}
\tilde{g}(X,Y)=g(X,LY),\qquad X,Y\in TM,
\end{gather*}
def\/ines a Riemannian metric on $M$. We call $L$ the {\it transfer tensor} between $g$ and~$\tilde g$. Notice, that the projection $\pi_P\colon P\to M$ is a Riemannian submersion with respect to $g_{{\rm SO}(M)}$ on~$P$ and~$\tilde{g}$ on~$M$. Denote by $\tilde{\nabla}$ the Levi-Civita connection of~$\tilde{g}$. One can show that \cite{gm}
\begin{gather*}
2g\big(\tilde{\nabla}_XY-\nabla_XY,LZ\big)=g((\nabla_XL)Y,Z)+g((\nabla_YL)X,Z)-g(X,(\nabla_ZL)Y).
\end{gather*}
Thus by Proposition \ref{prop:PropertiesofL} we get
\begin{gather}
2g\big(\tilde{\nabla}_XY-\nabla_XY,LZ\big) ={\bf B}((\nabla_X\xi)_Y+(\nabla_Y\xi)_X,\xi_Z)\nonumber\\
\hphantom{2g\big(\tilde{\nabla}_XY-\nabla_XY,LZ\big) =}{} +{\bf B}((\nabla_X\xi)_Z-(\nabla_Z\xi)_X,\xi_Y)+{\bf B}((\nabla_Y\xi)_Z-(\nabla_Z\xi)_Y,\xi_X).\label{eq:differencetensor}
\end{gather}

\section[Geometry of $G$-structures]{Geometry of $\boldsymbol{G}$-structures}\label{section5}

In this section we study the geometry of a $G$-structure $P$ in ${\rm SO}(M)$. In this case we need to consider the Levi-Civita connection $\nabla^{{\rm SO}(M)}$ of the Riemannian metric~$g_{{\rm SO}(M)}$. One can show~\cite{kn} that
\begin{alignat*}{3}
& \nabla^{{\rm SO}(M)}_{X^h}X^h =(\nabla_XY)^h-\frac{1}{2}R(X,Y)^{\ast},\qquad && \nabla^{{\rm SO}(M)}_{\alpha^{\ast}}Y^h =\frac{1}{2}R_{\alpha}(Y)^h,&\\
& \nabla^{{\rm SO}(M)}_{X^h}\alpha^{\ast} =\frac{1}{2}R_{\alpha}(X)^h+(\nabla_X\alpha)^{\ast},\qquad &&
\nabla^{{\rm SO}(M)}_{\alpha^{\ast}}\beta^{\ast} =-\frac{1}{2}[\alpha,\beta]^{\ast},&
\end{alignat*}
for $X,Y\in \Gamma(TM)$ and $\alpha,\beta\in\mathfrak{so}(M)$, or equivalently, $\alpha,\beta\in \mathfrak{so}(n)_{{\rm SO}(M)}$, where $R_{\alpha}$ is def\/ined by~\eqref{eq:defRT}.

It is convenient to f\/ind orthogonal projections
\begin{gather*}
T{\rm SO}(M)\mapsto TPq\quad\textrm{and}\qquad T{\rm SO}(M)\mapsto T^{\bot}P,
\end{gather*}
in order to derive the Levi-Civita connection and second fundamental form for $P$ in ${\rm SO}(M)$.

\begin{Proposition}\label{prop:ProjtoTP}
Let $X\in TM$ and $\alpha\in\mathfrak{so}(M)$.
\begin{enumerate}\itemsep=0pt
\item[$1.$] The orthogonal projection $T{\rm SO}(M)\mapsto TP$ equals
\begin{gather*}
X^h\mapsto L^{-1}(X)^{h'},\qquad
\alpha^{\ast}\mapsto \alpha_{\mathfrak{g}}^{\ast}-L^{-1}(\xi\cdot\alpha_{\mathfrak{m}})^{h'}.
\end{gather*}
\item[$2.$] The orthogonal projection $T{\rm SO}(M)\mapsto T^{\bot}P$ equals
\begin{gather*}
X^h\mapsto -(\xi_{L^{-1}(X)})^{\ast}-(\xi\cdot\xi_{L^{-1}X})^h,
\qquad
\alpha^{\ast}\mapsto \alpha_{\mathfrak{m}}^{\ast}
+(\xi_{L^{-1}(\xi\cdot\alpha_{\mathfrak{m}})})^{\ast}
+L^{-1}(\xi\cdot\alpha_{\mathfrak{m}})^h.
\end{gather*}
\end{enumerate}
\end{Proposition}
\begin{proof}
We will use frequently \eqref{eq:intrinsictorsionproperties} and \eqref{eq:xidotT}. Recall that $TP=\mathcal{H}'\oplus\mathfrak{m}_P$. For any $Y\in TM$
\begin{gather*}
g_{{\rm SO}(M)}\big(X^h-L^{-1}(X)^{h'},Y^{h'}\big) =g_{{\rm SO}(M)}\big(X^h-L^{-1}(X)^h-\xi_{L^{-1}(X)}^{\ast},Y^h+\xi_Y\big)\\
\hphantom{g_{{\rm SO}(M)}\big(X^h-L^{-1}(X)^{h'},Y^{h'}\big)}{} =g\big(X-L^{-1}(X),Y\big)-{\bf B}\big(\xi_{L^{-1}(X)},\xi_Y\big)=0
\end{gather*}
and for any $\beta\in\mathfrak{g}_P$
\begin{gather*}
g_{{\rm SO}(M)}\big(X^h-L^{-1}(X)^{h'},\beta^{\ast}\big)=-{\bf B}\big(\xi_{L^{-1}(X)},\beta\big)=0.
\end{gather*}
Clearly $L^{-1}(X)$ is tangent to $P$ and $X^h-L^{-1}(X)^{h'}=-(\xi_{L^{-1}(X)})^{\ast}-(\xi\cdot\xi_{L^{-1}X})^h$, which proves the desired decomposition for $X^h$. Analogously, \begin{gather*}
\alpha^{\ast}-\big(\alpha_{\mathfrak{g}}^{\ast}-L^{-1}(\xi\cdot\alpha_{\mathfrak{m}})^{h'}\big)=
\alpha_{\mathfrak{m}}^{\ast}+L^{-1}(\xi\cdot\alpha_{\mathfrak{m}})^{h'},
\end{gather*}
and
\begin{gather*}
g_{{\rm SO}(M)}\big(\alpha_{\mathfrak{m}}^{\ast}+L^{-1}(\xi\cdot\alpha_{\mathfrak{m}})^{h'},Y^{h'}\big) =
{\bf B}(\alpha_{\mathfrak{m}},\xi_Y)+g\big(L^{-1}(\xi\cdot\alpha_{\mathfrak{m}}),Y\big)+{\bf B}\big(\xi_{L^{-1}(\xi\cdot\alpha_{\mathfrak{m}})},\xi_Y\big)\\
\hphantom{g_{{\rm SO}(M)}\big(\alpha_{\mathfrak{m}}^{\ast}+L^{-1}(\xi\cdot\alpha_{\mathfrak{m}})^{h'},Y^{h'}\big)}{}
=-g(\xi\cdot\alpha_{\mathfrak{m}},Y)
+g\big(L^{-1}(\xi\cdot\alpha_{\mathfrak{m}},Y)\big)\!-g\big(\xi\cdot\xi_{L^{-1}(\xi\cdot\alpha_{\mathfrak{m}})},Y\big)\!\\
\hphantom{g_{{\rm SO}(M)}\big(\alpha_{\mathfrak{m}}^{\ast}+L^{-1}(\xi\cdot\alpha_{\mathfrak{m}})^{h'},Y^{h'}\big)}{}=0.
\end{gather*}
Moreover, for $\beta\in\mathfrak{g}_P$
\begin{gather*}
g_{{\rm SO}(M)}\big(\alpha_{\mathfrak{m}}^{\ast}+L^{-1}(\xi\cdot\alpha_{\mathfrak{m}})^{h'},\beta^{\ast}\big)=
{\bf B}(\alpha_{\mathfrak{m}},\beta)
+{\bf B}\big(\xi_{L^{-1}(\xi\cdot\alpha_{\mathfrak{m}})},\beta\big)=0.
\end{gather*}
Proposition follows from the fact that
\begin{gather*}
L^{-1}(\xi\cdot\alpha_{\mathfrak{m}})^{h'}=L^{-1}(\xi\cdot\alpha_{\mathfrak{m}})^h+\big(\xi_{L^{-1}(\xi\cdot\alpha_{\mathfrak{m}})}\big)^{\ast}
\end{gather*}
and $\alpha_{\mathfrak{g}}^{\ast}-L^{-1}(\xi\cdot\alpha_{\mathfrak{m}})^{h'}$ is tangent to $P$.
\end{proof}

\subsection{Intrinsic geometry}

For the intrinsic geometry we will compute the Levi-Civita connection $\nabla^P$ of $(P,g_{{\rm SO}(M)})$, the curvature tensor, Ricci tensor and sectional and scalar curvatures. We will compare geometry of $P$ with the geometry of the base manifold~$M$.

Let us f\/irst introduce operator $Q_{\alpha}$ and establish some relations. For $X\in TM$ and $\alpha\in\mathfrak{so}(M)$ put
\begin{gather*}
Q_{\alpha}(X)=L^{-1}(R_{\alpha}(X)-\xi\cdot(\nabla_X\alpha)_{\mathfrak{m}}).
\end{gather*}
Moreover, it is easy to see that for any $\alpha\in\mathfrak{so}(M)$ we have the following relations with respect to the decomposition $\mathfrak{so}(M)=\mathfrak{g}(M)\oplus\mathfrak{m}(M)$:
\begin{gather}\label{eq:basicrelations}
\begin{split}
&\nabla_X\alpha_{\mathfrak{g}} =\nabla'_X\alpha_{\mathfrak{g}}+[\xi_X,\alpha_{\mathfrak{g}}],\\
&\nabla_X\alpha_{\mathfrak{m}} =[\xi_X,\alpha_{\mathfrak{m}}]_{\mathfrak{g}}+\big( \nabla'_X\alpha_{\mathfrak{m}}+[\xi_X,\alpha_{\mathfrak{m}}]_{\mathfrak{m}} \big).
\end{split}
\end{gather}

\begin{Theorem}\label{thm:nablaP}
The Levi-Civita connection $\nabla^P$ of $P$ with the induced Riemannian metric $g_{{\rm SO}(M)}$ from ${\rm SO}(M)$ is of the following form
\begin{alignat*}{3}
& \nabla^P_{X^{h'}}Y^{h'} =(\tilde\nabla_XY)^{h'}-\frac{1}{2}R'(X,Y)^{\ast},\qquad && \nabla^P_{\alpha^{\ast}}Y^{h'} =\frac{1}{2}Q_{\alpha}(Y)^{h'},&\\
& \nabla^P_{X^{h'}}\beta^{\ast} =\frac{1}{2}Q_{\beta}(X)^{h'}+(\nabla'_X\beta)^{\ast},\qquad && \nabla^P_{\alpha^{\ast}}\beta^{\ast} =-\frac{1}{2}[\alpha,\beta]^{\ast},&
\end{alignat*}
where $X,Y\in TM$, $\alpha,\beta\in\mathfrak{g}_P$.
\end{Theorem}
\begin{proof}
First, by \eqref{eq:xidotT} and \eqref{eq:R_Tproperty} we have
\begin{gather*}
\tilde{g}(Q_{\alpha}(X),Z)={\bf B}(R(X,Z),\alpha)+{\bf B}(\nabla_X\alpha,\xi_Y).
\end{gather*}
Thus, using \eqref{eq:Rdecomp2} and \eqref{eq:differencetensor}
\begin{gather*}
\tilde{g}(Q_{\xi_X}(Y)+Q_{\xi_Y}(X)+\xi\cdot\xi_{\nabla_XY+\nabla_YX},Z)\\
\qquad {}= {\bf B}((\nabla_X\xi)_Z-(\nabla_Z\xi)_X,\xi_Y)+{\bf B}((\nabla_Y\xi)_Z-(\nabla_Z\xi)_Y,\xi_X)-{\bf B}([\xi_X,\xi_Z],\xi_Y)\\
\qquad \quad{}-{\bf B}([\xi_Y,\xi_Z],\xi_X)+{\bf B}((\nabla_Y\xi)_X+(\nabla_X\xi)_Y,\xi_Z)\\
\qquad {} =2\tilde{g}(\tilde{\nabla}_XY-\nabla_XY,Z).
\end{gather*}
We have shown that
\begin{gather}\label{eq:nablaPproof}
P(X,Y)=\frac{1}{2}\big(Q_{\xi_X}(Y)+Q_{\xi_Y}(X)+\xi\cdot\xi_{\nabla_XY+\nabla_YX}\big).
\end{gather}
By the formula for the connection $\nabla^{{\rm SO}(M)}$
\begin{gather*}
\nabla^P_{X^{h'}}Y^{h'} = \nabla^{{\rm SO}(M)}_{X^h}Y^h+\nabla^{{\rm SO}(M)}_{(\xi_X)^{\ast}}Y^h
+\nabla^{{\rm SO}(M)}_{X^h}(\xi_Y)^{\ast}+\nabla^{{\rm SO}(M)}_{(\xi_X)^{\ast}}(\xi_Y)^{\ast}\\
\hphantom{\nabla^P_{X^{h'}}Y^{h'}}{}
=\left(\nabla_XY+\frac{1}{2}R_{\xi_Y}(X)+\frac{1}{2}R_{\xi_X}(Y)\right)^h+\left(-\frac{1}{2}R(X,Y)+\nabla_X\xi_Y-\frac{1}{2}[\xi_X,\xi_Y]\right)^{\ast}.
\end{gather*}
Since, by relations \eqref{eq:curvaturedecomp} between intrinsic torsion and curvature tensor,
\begin{gather*}
\left(-\frac{1}{2}R(X,Y)+\nabla_X\xi_Y-\frac{1}{2}[\xi_X,\xi_Y]\right)_{\mathfrak{g}}=-\frac{1}{2}R'(X,Y)
\end{gather*}
and
\begin{gather*}
\left(-\frac{1}{2}R(X,Y)+\nabla_X\xi_Y-\frac{1}{2}[\xi_X,\xi_Y]\right)_{\mathfrak{m}}=\frac{1}{2}\left( \nabla_X\xi_Y+\nabla_Y\xi_X+\xi_{[X,Y]} \right),
\end{gather*}
then, with the use of Proposition~\ref{prop:ProjtoTP}, we get
\begin{gather*}
\nabla^P_{X^{h'}}Y^{h'}=L^{-1}\left(\frac{1}{2}R_{\xi_Y}(X)+\frac{1}{2}R_{\xi_X}(Y)-\frac{1}{2}\xi\cdot(\nabla_X\xi_Y+\nabla_Y\xi_X)\right)^{h'}\\
\hphantom{\nabla^P_{X^{h'}}Y^{h'}=}{}
+L^{-1}\left(\nabla_XY-\frac{1}{2}\xi\cdot\xi_{[X,Y]}\right)^{h'}-\frac{1}{2}R'(X,Y)^{\ast}.
\end{gather*}
Moreover $\nabla_XY-\frac{1}{2}\xi\cdot\xi_{[X,Y]}=L(\nabla_XY)+\xi\cdot\xi_{\nabla_XY+\nabla_YX}$. Thus, by the def\/inition of $Q_{\alpha}$ and formula~\eqref{eq:nablaPproof} we get the desired formula for $\nabla^P_{X^{h'}}Y^{h'}$.

For the proof of the second formula we will use \eqref{eq:basicrelations} and Proposition~\ref{prop:ProjtoTP}. We have
\begin{gather*}
\nabla^P_{X^{h'}}\beta^{\ast} =\left( \nabla^{{\rm SO}(M)}_{X^h}\beta^{\ast}
+\nabla^{{\rm SO}(M)}_{(\xi_X)^{\ast}}\beta^{\ast} \right)^{\top} =\left(\frac{1}{2}R_{\beta}(X)^h+(\nabla_X\beta)^{\ast}-\frac{1}{2}[\xi_X,\beta]^{\ast}\right)^{\top}\\
\hphantom{\nabla^P_{X^{h'}}\beta^{\ast}}{}
=L^{-1}\left(\frac{1}{2}R_{\beta}(X)-\xi\cdot(\nabla_X\beta)_{\mathfrak{m}}+\frac{1}{2}\xi\cdot[\xi_X,\beta]_{\mathfrak{m}}\right)^{h'}
+(\nabla_X\beta)^{\ast}_{\mathfrak{g}}\\
\hphantom{\nabla^P_{X^{h'}}\beta^{\ast}}{}
=Q_{\beta}(X)^{h'}+(\nabla'_X\beta)^{\ast}.
\end{gather*}
We prove the remaining relations analogously.
\end{proof}

\begin{Proposition}\label{prop:RP}
The curvature tensor $R^P$ of $\nabla^P$ equals
\begin{gather*}
R^P\big(X^{h'},Y^{h'}\big)Z^{h'} =\big(\tilde{R}(X,Y)Z\big)^{h'}-\frac{1}{4}\big( Q_{R'(Y,Z)}(X)-Q_{R'(X,Z)}(Y)-2Q_{R(X,Y)}(Z)\big)^{h'}\\
\hphantom{R^P\big(X^{h'},Y^{h'}\big)Z^{h'} =}{}-\frac{1}{2}(D_XR')(Y,Z)^{\ast}+\frac{1}{2}(D_YR')(X,Z)^{\ast},\\
R^P\big(X^{h'},Y^{h'}\big)\gamma^{\ast} =\frac{1}{2}\big((D_XQ)_{\gamma}(Y)-(D_YQ)_{\gamma}(X)\big)^{h'}+\frac{1}{2}[R'(X,Y),\gamma]^{\ast}\\
\hphantom{R^P\big(X^{h'},Y^{h'}\big)\gamma^{\ast} =}{}
-\frac{1}{4}\big( R'(X,Q_{\gamma}(Y))-R'(Y,Q_{\gamma}(X)) \big)^{\ast},\\
R^P\big(X^{h'},\beta^{\ast}\big)Z^{h'} =\frac{1}{2}(D_XQ)_{\beta}(Z)^{h'}-\frac{1}{4}\big( R'(X,Q_{\beta}(Z))+[\beta,R'(X,Z)] \big)^{\ast},\\
R^P\big(X^{h'},\beta^{\ast}\big)\gamma^{\ast} =-\frac{1}{4}\big( Q_{[\beta,\gamma]}(X)+Q_{\beta}(Q_{\gamma}(X)) \big)^{h'},\\
R^P(\alpha^{\ast},\beta^{\ast})Z^{h'} =\frac{1}{4}[Q_{\alpha},Q_{\beta}](Z)^{h'}+\frac{1}{2}Q_{[\alpha,\beta]}(Z)^{h'},\\
R^P(\alpha^{\ast},\beta^{\ast})\gamma^{\ast} =-\frac{1}{4}[[\alpha,\beta],\gamma]^{\ast},
\end{gather*}
where
\begin{gather*}
(D_XR')(Y,Z) =\nabla '_XR'(Y,Z)-R'(\tilde{\nabla}_XY,Z)-R'(Y,\tilde{\nabla}_XZ),\\
(D_XQ)_{\alpha}(Y) =\tilde{\nabla}_XQ_{\alpha}(Y)-Q_{\nabla '_X\alpha}(Y)-Q(\tilde{\nabla}_XY).
\end{gather*}
\end{Proposition}
\begin{proof}
Follows directly by Theorem \ref{thm:nablaP}, Jacobi identity and the relations
\begin{gather*}
R'(X,Y)\gamma=[R'(X,Y),\gamma],\qquad \nabla '_X[\beta,\gamma]=[\nabla '_X\beta,\gamma]+[\beta,\nabla '_X\gamma].
\end{gather*}
Notice that $[X^{h'},\beta^{\ast}]=(\nabla '_X\beta)^{\ast}$ and $[\alpha^{\ast},\beta^{\ast}]=-[\alpha,\beta]^{\ast}$.
\end{proof}

\begin{Corollary}\label{cor:RicciP}
The Ricci curvature tensor ${\rm Ric}^P$ equals
\begin{gather*}
{\rm Ric}^P\big(X^{h'},Y^{h'}\big) =\widetilde{{\rm Ric}}(X,Y)-\frac{3}{4}\sum_i{\bf B}\big(R'(X,\tilde{e_i}),R'(Y,\tilde{e_i})\big)
+\frac{1}{4}\sum_A \tilde{g}(Q_{\alpha_A}(X),Q_{\alpha_A}(Y)),\\
{\rm Ric}^P(X^{h'},\gamma^{\ast}) =\frac{1}{2}\big( \big(\widetilde{\operatorname{div}}\, Q\big)_{\gamma}(X)-\operatorname{tr}_{\tilde{g}}(D_XQ)_{\gamma} \big),\\
{\rm Ric}^P(\beta^{\ast},\gamma^{\ast}) =\frac{1}{4}\left( \sum_i \tilde{g}\big(Q_{\beta}(\tilde{e_i}),Q_{\gamma}(\tilde{e_i})\big)+\sum_A {\bf B}([\alpha_A,\beta],[\alpha_A,\gamma]) \right),
\end{gather*}
where
\begin{gather*}
\big(\widetilde{\operatorname{div}}\,  Q\big)_{\gamma}(X)=\sum_i\tilde{g}\big((D_{\tilde e_i}Q)_{\gamma}(X),\tilde{e_i}\big)
\end{gather*}
and $(\tilde{e_i})$ is an orthonormal basis with respect to $\tilde{g}$ on $M$ and $(\alpha_A)$ is an orthonormal basis of $\mathfrak{g}_P$ with respect to ${\bf B}$.
\end{Corollary}
\begin{proof}
First, notice that
\begin{gather}\label{eq:QalphaRprime}
\tilde{g}(Q_{\alpha}(X),Y)={\bf B}(R'(X,Y),\alpha),
\end{gather}
which follows by the def\/inition of $Q_{\alpha}$, \eqref{eq:basicrelations} and relations between intrinsic torsion and curvature tensor \eqref{eq:curvaturedecomp}. Hence $Q_{\alpha}$ is skew-symmetric with respect to $\tilde{g}$. Now, it suf\/f\/ices to use the formulas for the curvature tensor~$R^P$.
\end{proof}

\begin{Corollary}\label{cor:SectionalP}
The sectional curvatures of $\nabla^P$ are given by the following formulas
\begin{gather*}
\kappa^P\big(X^{h'},Y^{h'}\big) =\tilde{\kappa}(X,Y)-\frac{3}{4}\|R'(X,Y)\|^2,\\
\kappa^P\big(X^{h'},\beta^{\ast}\big) =\frac{1}{4}\|Q_{\beta}(X)\|^2_{\tilde{g}},\qquad
\kappa^P(\alpha^{\ast},\beta^{\ast}) =\frac{1}{4}\|[\alpha,\beta]\|^2,
\end{gather*}
where $X$, $Y$ are orthonormal with respect to $\tilde{g}$ and $\alpha,\beta\in\mathfrak{g}_P$ orthonormal with respect to~${\bf B}$.
\end{Corollary}
\begin{proof}
First and last relations follow immediately by the formulas for the curvature tensor $R^P$ and by \eqref{eq:QalphaRprime}. For the proof of the second one it suf\/f\/ices to use the skew-symmetry of $Q_{\beta}$ with respect to $\tilde{g}$ (see the proof of Corollary \ref{cor:RicciP}).
\end{proof}

\begin{Corollary}\label{cor:scalarP}
The scalar curvature of $\nabla^P$ equals
\begin{gather*}
s^P=\tilde{s}-\frac{3}{4}\sum_{i,j}\|R'(\tilde{e_i},\tilde{e_j})\|^2+\frac{1}{2}\sum_{i,a}\|Q_{\alpha_A}(\tilde{e_i})\|^2_{\tilde{g}}
+\frac{1}{4}\sum_{A,B}\|[\alpha_A,\alpha_B]\|^2.
\end{gather*}
\end{Corollary}

By the above formulas for the curvatures we have the following relations between the geometry of $(P,g_{{\rm SO}(M)})$ and $M$.

\begin{Corollary}\label{cor:relPandtildeM}
We have:
\begin{enumerate}\itemsep=0pt
\item[$1.$] If $\dim M>2$, then $P$ is never of constant sectional curvature.
\item[$2.$]  If the intrinsic torsion $\xi$ vanishes and $(M,g)$ is of constant sectional curvature $0\leq \kappa\leq\frac{2}{3}$, then $P$ has non-negative sectional curvatures.
\item[$3.$]  ${\rm Ric}^P(\alpha^{\ast})$ is non-negative in any direction $\alpha\!\in\! \mathfrak{g}_P$. Moreover, if $\widetilde{\rm Ric}(X)\!\geq\! \frac{3}{4}\!\sum_i\|R'(X,\tilde{e_i})\|^2$, then ${\rm Ric}^P(X)$ is non-negative.
\item[$4.$]  If $R'=0$ and scalar curvature of $(M,\tilde{g})$ is positive, then the scalar curvature of $P$ is positive.
\end{enumerate}
\end{Corollary}

\subsection{Extrinsic geometry}

The properties of the extrinsic geometry are encoded in the second fundamental form, which we will derive explicitly. Moreover, we will compute the mean curvature vector of $P$ in ${\rm SO}(M)$ and relate the minimality of $P$ with the harmonicity of an induced section $\sigma$ with appropriate Riemannian structure.

Adopt the notation from the previous sections. For $\alpha\in\mathfrak{m}_P$ let
\begin{gather*}
\alpha^{+}=\alpha^{\ast}+(\xi\cdot\alpha)^h.
\end{gather*}
Then $\alpha^{+}\in T^{\bot}P$.

\begin{Theorem}\label{thm:secfundformP}
The second fundamental form $\Pi^P$ of $P$ in ${\rm SO}(M)$ satisfies the following relations
\begin{gather*}
g_{{\rm SO}(M)}\big(\Pi^P(X^{h'},Y^{h'}),\alpha^{+}\big) =\frac{1}{2}{\bf B}((\nabla_X\xi)_Y+(\nabla_Y\xi)_X-\xi_{R_{\xi_X}(Y)+R_{\xi_Y}(X)},\alpha),\\
g_{{\rm SO}(M)}\big(\Pi^P(X^{h'},\gamma^{\ast}),\alpha^{+}\big) =\frac{1}{2}{\bf B}([\xi_X,\gamma]_{\mathfrak{m}}-\xi_{R_{\gamma}(X)},\alpha),\\
g_{{\rm SO}(M)}\big(\Pi^P(\beta^{\ast},\gamma^{\ast}),\alpha^{+}\big)=0.
\end{gather*}
\end{Theorem}
\begin{proof}
By the relations \eqref{eq:curvaturedecomp} (see also proof of Theorem \ref{thm:nablaP}) we have
\begin{gather*}
\nabla^{{\rm SO}(M)}_{X^h}Y^h  =\left( \nabla_XY+\frac{1}{2}R_{\xi_X}(Y)+\frac{1}{2}R_{\xi_Y}(X) \right)^h\\
\hphantom{\nabla^{{\rm SO}(M)}_{X^h}Y^h  =}{} +\frac{1}{2}\big( \nabla_X\xi_Y+\nabla_Y\xi_X+\xi_{[X,Y]} \big)^{\ast}+\textrm{$\mathfrak{g}_P$-component}.
\end{gather*}
Thus
\begin{gather*}
g_{{\rm SO}(M)}\big(\Pi^P(X^{h'},Y^{h'}),\alpha^{+}\big) =g\left(\nabla_XY+\frac{1}{2}(R_{\xi_X}(Y)+R_{\xi_Y}(X)),\xi\cdot\alpha\right)\\
\hphantom{g_{{\rm SO}(M)}\big(\Pi^P(X^{h'},Y^{h'}),\alpha^{+}\big) =}{}
+\frac{1}{2}{\bf B}\big(\nabla_X\xi_Y+\nabla_Y\xi_X+\xi_{[X,Y]},\alpha\big),
\end{gather*}
which implies the f\/irst equality. Moreover,
\begin{gather*}
\nabla^{{\rm SO}(M)}_{X^{h'}}\gamma^{\ast}=\frac{1}{2}R_{\gamma}(X)^h+
\frac{1}{2}[\xi_X,\gamma]^{\ast}+\textrm{$\mathfrak{g}_P$-component},
\end{gather*}
which proves the second equality. Since $[\mathfrak{g},\mathfrak{g}]\subset\mathfrak{g}$, it follows that $\mathfrak{m}_P$-component of $\nabla^{{\rm SO}(M)}_{\alpha^{\ast}}\beta^{\ast}$ vanishes.
\end{proof}

By above theorem we get the following implication.

\begin{Corollary}\label{cor:totgeodP}
If a $G$-structure $P$ is integrable, i.e., the intrinsic torsion $\xi$ vanishes, then $P$ is totally geodesic in ${\rm SO}(M)$.
\end{Corollary}

Now, we can state the main theorem of this article.

\begin{Theorem}\label{thm:minimalityofP}
A $G$-structure $P$ is minimal in ${\rm SO}(M)$ if and only the induced section \mbox{$\sigma\colon M{\to} N$} is a harmonic map with respect to Riemannian metrics $\tilde{g}$ and $\skal{\cdot}{\cdot}$ on $M$ and $N$, respectively.
\end{Theorem}
\begin{proof}
Recall, that by Proposition~\ref{prop:harmonicsigma1} section $\sigma\colon (M,\tilde{g})\to (N,\skal{\cdot}{\cdot})$ is a harmonic map if and only if
\begin{gather}
\sum_i(\nabla_{\tilde{e_i}}\xi)_{\tilde{e_i}}-\xi_{S(\tilde{e_i},\tilde{e_i})}=0,\label{eq:H1}\\ 
\sum_i R_{\xi_{\tilde{e_i}}}(\tilde{e_i})-S(\tilde{e_i},\tilde{e_i})=0,\label{eq:H2} 
\end{gather}
whereas, by Theorem \ref{thm:secfundformP}, minimality of $P$ is equivalent to the following condition
\begin{gather}\label{eq:minimalityP}
\sum_i (\nabla_{\tilde{e_i}}\xi)_{\tilde{e_i}}-\xi_{R_{\xi_{\tilde{e_i}}}(\tilde{e_i})}=0.
\end{gather}
Clearly, \eqref{eq:H1} and \eqref{eq:H2} imply \eqref{eq:minimalityP}. Conversely, assume \eqref{eq:minimalityP} holds. It suf\/f\/ices to show that~\eqref{eq:H2} holds. By~\eqref{eq:differencetensor} and~\eqref{eq:minimalityP} we have
\begin{gather*}
g\left( \sum_i R_{\xi_{\tilde{e_i}}}(\tilde{e_i})-S(\tilde{e_i},\tilde{e_i}),LZ\right)\\
\qquad{}=-\sum_i{\bf B}\big((\nabla_{e_i}\xi)_{e_i},\xi_Z\big)
+{\bf B}\big((\nabla_{e_i}\xi)_Z-(\nabla_Z\xi)_{e_i},\xi_{e_i}\big) +\sum_i g\big(R_{\xi_{e_i}}(e_i),Z-\xi\cdot \xi_Z\big)\\
\qquad{} =-\sum_i{\bf B}\big((\nabla_{e_i}\xi)_Z-(\nabla_Z\xi)_{e_i},\xi_{e_i}\big)+{\bf B}\big(R(e_i,Z),\xi_{e_i}\big).
\end{gather*}
By the relations \eqref{eq:curvaturedecomp} and the fact that ${\rm SO}(n)/G$ is a normal homogeneous space we get the desired equality.
\end{proof}

\begin{Remark}\label{rem:H1impliesH2}
Let us comment on Theorem \ref{thm:minimalityofP}.
\begin{enumerate}\itemsep=0pt
\item Two equivalent conditions of Theorem \ref{thm:minimalityofP} are also equivalent to minimality the ima\-ge~$\sigma(M)$ inside $N$. This follows by the general facts concerning harmonic maps~\cite{el}.
\item As it was pointed out by an anonymous referee, condition~\eqref{eq:H1} implies~\eqref{eq:H2} (for $\tilde{g}=\sigma^{\ast}\skal{\,}{\,}$). The proof follows by similar arguments as in the proof of \cite[Proposition~2.5]{ggv}. The author wishes to thank anonymous referee for this observation.
\end{enumerate}
\end{Remark}

\section{Some examples concerning almost product structures}\label{section6}

In this section we illustrate obtained results for the ${\rm SO}(m)\times {\rm SO}(n-m)$-structures, often called almost product structures. The case of other possible $G$-structures, for example coming from the Berger list of possible Riemannian holonomy groups, will be studied by the author independently. The considered examples have been already studied in a similar context \cite{ggv,gd,hl}. Thus we only list them stating theirs relevant properties.

Let $(M,g)$ be an oriented $n$-dimensional Riemannian manifold and let $G\subset {\rm SO}(n)$ be the closed subgroup of the form
\begin{gather*}
G={\rm SO}(m)\times {\rm SO}(n-m)
\end{gather*}
for some $m=1,\ldots,n-1$. The quotient ${\rm SO}(n)/G$ is a symmetric space, which is the oriented Grassmannian $\operatorname{Gr}^{\rm o}_m(\mathbb{R}^n)$ of oriented $m$-dimensional subspaces in the Euclidean space $\mathbb{R}^n$. The reduction of the oriented orthonormal frame bundle ${\rm SO}(M)$ to the subbundle $P$ with the structure group~$G$ is equivalent to the existence of $m$-dimensional distribution $E$ on $M$ and, hence, its orthogonal complement $F=E^{\bot}$. We call $M$ with the distinguished distribution $E$ an almost product structure.

The connection $\nabla'$ induced by the connection form $\omega_{\mathfrak{g}}$, where $\omega$ is the connection form of the Levi-Civita connection $\nabla$, takes the form
\begin{gather*}
\nabla '_XY=\big(\nabla_XY^{\top}\big)^{\top}+\big(\nabla_XY^{\bot}\big)^{\bot},
\end{gather*}
where the decomposition $X=X^{\top}+X^{\bot}$ is taken with respect to $TM=E\oplus F$. In other words it is the sum of two connections induced by $\nabla$-connections in the vector bundles $E$ and $F$ over~$M$. The intrinsic torsion of almost product structure equals
\begin{gather*}
\xi_XY=\big(\nabla_XY^{\bot}\big)^{\top}+\big(\nabla_XY^{\top}\big)^{\bot}.
\end{gather*}
The associated bundle $N={\rm SO}(M)/G={\rm SO}(M)\times_{{\rm SO}(n)}({\rm SO}(n)/G)$ is the Grassmann bundle $\operatorname{Gr}^{\rm o}_m(TM)$ of $m$-dimensional oriented subspaces of tangent spaces to $M$ and the induced section $\sigma\colon M\to N$ is just the Gauss map of the distribution $E$. Therefore the main result (Theorem~\ref{thm:minimalityofP}) states that a ${\rm SO}(m)\times {\rm SO}(n-m)$-structure $P\subset {\rm SO}(M)$ is minimal if and only if the Gauss map $\sigma\colon (M,\tilde{g})\to (\operatorname{Gr}^{\rm o}_m(TM),\skal{\cdot}{\cdot})$ of $E$ is a harmonic map or, alternatively, if the image $E=\sigma(M)\subset \operatorname{Gr}^{\rm o}_m(TM)$ is minimal in the oriented Grassmann bundle $\operatorname{Gr}^{\rm o}_m(TM)$.

Let us now give two examples (compare~\cite{ggv,gd,hl}).

\begin{Example}
Let $(M,g,X_0)$ be a $K$-contact manifold with the Reeb vector f\/ield $X_0$, i.e., $X_0$~is a~unit Killing vector f\/ield and there exist one-form $\eta$ and endomorphism $\varphi$ such that
\begin{gather*}
\eta(X)=g(X,X_0),\qquad \varphi^2X=-X+\eta(X)X_0,\qquad d\eta(X,Y)=g(X,\varphi Y),\qquad \iota_{X_0}d\eta=0
\end{gather*}
for all $X,Y\in\Gamma(TM)$.

The one-dimensional distribution $E$ tangent to $X_0$ def\/ines the almost product structure on~$M$. Since $X_0$ is geodesic vector f\/ield, the distribution $E$ is totally geodesic. It can be shown, that $E\subset \operatorname{Gr}^{\rm o}_{n-1}(TM)$ is a minimal submanifold or, in other words, the immersion $\sigma\colon (M,\tilde{g})\to (\operatorname{Gr}^{\rm o}_{n-1}(TM),\skal{\cdot}{\cdot})$, $\sigma(x)=E_x$, $x\in M$, is minimal. Therefore, the ${\rm SO}(n-1)\times {\rm SO}(1)$-structure~$P$ induced by~$E$ is minimal in the orthonormal frame bundle~${\rm SO}(M)$.
\end{Example}

\begin{Example}
Consider the sphere $S^{4n-1}$ in the Euclidean space $\mathbb{R}^{4n}$ and let $I$, $J$, $K$ be the usual quaternionic structure on $\mathbb{R}^{4n}$. The $3$-dimensional subspaces spanned by $IN$, $JN$, $KN$, where $N$ is the unit outward vector f\/ield to $S^{4n-1}$, determine the Hopf distribution $E$ on $S^{4n-1}$. Since the Hopf distribution def\/ines the minimal immersion of $S^{4n-1}$ into the Grassmann bundle $\operatorname{Gr}^{\rm o}_3(TS^{4n-1})$, it follows that the ${\rm SO}(3)\times {\rm SO}(4n-4)$-structure $P\subset {\rm SO}(S^{4n-1})$ is minimal.
\end{Example}

\vspace{-2mm}

\pdfbookmark[1]{References}{ref}
\LastPageEnding

\end{document}